\documentclass[11pt]{amsart}

\author{Carlo Sanna}
\address{Universit\`a degli Studi di Torino\\Department of Mathematics\\Turin, Italy}
\email{carlo.sanna.dev@gmail.com}

\title{On the closure of the image of the generalized~divisor~function}

\usepackage{amsmath}
\usepackage{amssymb}
\usepackage{amsthm}
\usepackage{geometry}
\usepackage[usenames]{color}
\usepackage{hyperref}
\usepackage{enumerate}
\hypersetup{colorlinks=true}
\usepackage{algpseudocode}
\usepackage{setspace}
\usepackage{graphicx}
\DeclareGraphicsExtensions{.png}

\newtheorem{thm}{Theorem}[section]

\newtheorem{lem}[thm]{Lemma}
\theoremstyle{remark}

\def\Cl{\operatorname{Cl}}
                                 
\begin{document}

\begin{abstract}
For any real number $s$, let $\sigma_s$ be the generalized divisor function, i.e., the arithmetic function defined by $\sigma_s(n) := \sum_{d \, \mid \, n} d^s$, for all positive integers $n$.
We prove that for any $r > 1$ the topological closure of $\sigma_{-r}(\mathbf{N}^+)$ is the union of a finite number of pairwise disjoint closed intervals $I_1, \ldots, I_\ell$.
Moreover, for $k=1,\ldots,\ell$, we show that the set of positive integers $n$ such that $\sigma_{-r}(n) \in I_k$ has a positive rational asymptotic density $d_k$.
In fact, we provide a method to give exact closed form expressions for $I_1, \ldots, I_\ell$ and $d_1, \ldots, d_\ell$, assuming to know $r$ with sufficient precision.
As an example, we show that for $r = 2$ it results $\ell = 3$, $I_1 = [1, \pi^2/9]$, $I_2 = [10/9, \pi^2/8]$, $I_3 = [5/4, \pi^2 / 6]$, $d_1 = 1/3$, $d_2 = 1/6$, and $d_3 = 1/2$.
\end{abstract}

\maketitle

\section{Introduction}

For any real number $s$, let the \emph{generalized divisor function} $\sigma_s$ be defined by 
\begin{equation*}
\sigma_s(n) := \sum_{d \, \mid \, n} d^s ,
\end{equation*}
for each positive integer $n$, where $d$ runs over all the positive divisors of $n$.
It is well-known that $\sigma_s$ is a multiplicative arithmetic function.
For $r > 1$ it can be readily proved that $\sigma_{-r}(\mathbf{N}^+) \subseteq [1, \zeta(r)[$, where $\zeta$ is the Riemann zeta function.
Moreover, Defant~\cite{Def14} proved that $\sigma_{-r}(\mathbf{N}^+)$ is dense in $[1, \zeta(r)[$ if and only if $r \leq \eta$, where $\eta \approx 1.8877909$ is the unique real number in $]1, 2]$ that satisfies the equation
\begin{equation*}
\frac{2^\eta}{2^\eta-1} \cdot \frac{3^\eta + 1}{3^\eta-1} = \zeta(\eta) .
\end{equation*}
In this paper, we shall prove the following result about the topological closure of $\sigma_{-r}(\mathbf{N}^+)$.

\begin{thm}\label{thm:closure}
For any $r > 1$, the topological closure of $\sigma_{-r}(\mathbf{N}^+)$ is the union of a finite number of pairwise disjoint closed intervals $I_1, \ldots, I_\ell$.
Furthermore, for $k=1,\ldots,\ell$, the set of positive integers $n$ such that $\sigma_{-r}(n) \in I_k$ has a positive rational asymptotic density $d_k$.
\end{thm}

In fact, we provide a method to give exact closed form expressions for $I_1, \ldots, I_\ell$ and $d_1, \ldots, d_\ell$, assuming to know $r$ with sufficient precision.
As an example, we show that for $r = 2$ it results $\ell = 3$, $I_1 = [1, \pi^2/9]$, $I_2 = [10/9, \pi^2/8]$, $I_3 = [5/4, \pi^2 / 6]$, $d_1 = 1/3$, $d_2 = 1/6$, and $d_3 = 1/2$.

Let us note that the existence of the densities $d_1, \ldots, d_\ell$ is already known.
In fact, by the Erd\H{o}s--Wintner theorem \cite[Chap. 5]{Ell79} it follows that the additive arithmetic function $\log \sigma_{-r}$ has a limiting distribution, and consequently so does $\sigma_{-r}$. 
Thus, for each real interval $I$ the set of positive integers $n$ such that $\sigma_{-r}(n) \in I$ has an asymptotic density.
Our new contributions are the proof that $d_1, \ldots, d_\ell$ are all positive and rational, and the method to compute them.

Lastly, we note that from a result of Defant~[Theorem~3.3]\cite{Def15} it follows that $\ell \to +\infty$ as $r \to +\infty$.
 
\subsection*{Notation}
For each positive integer $j$, we write $p_j$ for the $j$-th prime number and $\mathcal{N}_j$ for the set of all the positive integers without prime factors $ \leq p_j$, also $\mathcal{N}_0 := \mathbf{N}^+$.
We use $\infty$ in place of $+\infty$ and we define
\begin{equation*}
\sigma_{-r}(p^\infty) := \lim_{a \to \infty} \sigma_{-r}(p^a) = \frac1{1 - p^{-r}} ,
\end{equation*}
for all $r > 1$ and prime numbers $p$.
Moreover, $\upsilon_p$ will be the usual $p$-adic valuation for the prime number $p$.
For each $X \subseteq \mathbf{R}$, we denote by $\Cl(X)$ the closure of $X$, while $a \cdot X := \{ax : x \in X\}$ for any $a \in \mathbf{R}$.

\section{Preliminaries}

We begin with an inequality between consecutive prime numbers.

\begin{lem}\label{lem:next_prime_bound}
For each integer $j \geq 463$, we have $p_{j+1} \leq \left(1 + \frac1{2\log^2 p_j}\right) p_j$.
\end{lem}
\begin{proof}
See \cite[Proposition 1.10]{Dus98}.
\end{proof}

Now we recall the Euler product for the Riemann zeta function:
\begin{equation*}
\zeta(r) = \prod_{k=1}^\infty \frac1{1 - p_k^{-r}} ,
\end{equation*}
for any $r > 1$, which in particular implies that
\begin{equation}\label{equ:crucial_prod}
\prod_{k \, > \, j} \frac1{1 - p_k^{-r}} = \prod_{k=1}^j (1 - p_k^{-r}) \cdot \zeta(r),
\end{equation}
for each nonnegative integer $j$.
The product on the left-hand side of (\ref{equ:crucial_prod}) will play a crucial role in our proofs.

\begin{lem}\label{lem:prod_bound}
For any $r > 1$, there exist only finitely many positive integers $j$ such that
\begin{equation*}
\prod_{k \, > \, j} \frac1{1 - p_k^{-r}} < 1 + p_j^{-r} .
\end{equation*}
\end{lem}
\begin{proof}
For each positive integer $j$, put
\begin{equation*}
\Pi_j := \frac1{1 + p_j^{-r}} \prod_{k \, > \, j} \frac1{1 - p_k^{-r}} .
\end{equation*}
Thanks to Lemma~\ref{lem:next_prime_bound}, we know that there exists a positive integer 
\begin{equation}\label{equ:max}
j^\prime \leq \max\left\{463, \exp\left(\sqrt{\frac1{2(2^{1/r}-1)}}\right)\right\} ,
\end{equation}
such that for all integers $j \geq j^\prime$ it holds $p_{j+1} \leq 2^{1/r} p_j$, which in turn implies $p_{j+1}^{-r} > (2 p_j^r + 1)^{-1}$.
Therefore,
\begin{equation*}
\Pi_{j+1} := (1 + p_{j}^{-r}) \cdot \frac{1 - p_{j+1}^{-r}}{1 + p_{j+1}^{-r}} \cdot \Pi_j < (1 + p_{j}^{-r}) \cdot \frac{1 - (2p_j^r + 1)^{-1}}{1 + (2p_j^r + 1)^{-1}} \cdot \Pi_j = \Pi_j ,
\end{equation*}
for each integer $j \geq j^\prime$, since the function $[0,1] \to \mathbf{R} : x \mapsto (1-x)/(1+x)$ is monotone nonincreasing.
Thus $(\Pi_j)_{j \geq j^\prime}$ is a monotone nonincreasing sequence.
Now $\Pi_j \to 1$ as $j \to \infty$, so it follows that $\Pi_j \geq 1$ for all integers $j \geq j^\prime$, which is our claim.
\end{proof}

In light of Lemma~\ref{lem:prod_bound}, for each $r > 1$ let $j_0 = j_0(r)$ be the least nonnegative integer such that
\begin{equation}\label{equ:prod_ineq}
\prod_{k \, > \, j} \frac1{1 - p_k^{-r}} \geq 1 + p_j^{-r} ,
\end{equation}
for all the positive integers $j > j_0$.
Actually, the proof of Lemma~\ref{lem:prod_bound} explains how to effectively compute $j_0$:
First, find the minimal $j^\prime$ by checking the inequality $p_{j+1} \leq 2^{1/r} p_j$ for all the positive integers $j$ less than the right-hand side of (\ref{equ:max}).
Second, find $j_0$ by checking the inequality (\ref{equ:prod_ineq}) for all the positive integers less than $j^\prime$.

The next lemma is the first step to determine the closure of $\sigma_{-r}(\mathbf{N}^+)$.

\begin{lem}\label{lem:first_step}
For $r > 1$, we have
\begin{equation*}
\Cl(\sigma_{-r}(\mathcal{N}_{j_0})) = \left[1, \prod_{k \,>\, j_0} \frac1{1 - p_k^{-r}} \right] .
\end{equation*}
\end{lem}
\begin{proof}
Pick any $x \in \left[1, \prod_{k > j_0} \frac1{1 - p_k^{-r}} \right]$.
Define the sequences $(x_j)_{j \geq j_0}$ and $(a_j)_{j > j_0}$ via the following greedy process:
$x_{j_0} := 1$, while for each integer $j > j_0$ let $a_j$ be the greatest $a \in \mathbf{N} \cup \{\infty\}$ such that $x_{j-1} \cdot \sigma_{-r}(p_j^a) \leq x$, and put also $x_j := x_{j-1} \cdot \sigma_{-r}\big(p_j^{a_j}\big)$.
By construction, $(x_j)_{j \geq j_0}$ is a monotone nondecreasing sequence such that $x_j \leq x$, for each integer $j \geq j_0$.
In particular, there exists $\ell := \displaystyle \lim_{j \to \infty} x_j$ and it holds $\ell \leq x$.
Moreover, $x_j \in \Cl(\sigma_{-r}(\mathcal{N}_{j_0}))$ for any integer $j \geq j_0$, hence $\ell \in \Cl(\sigma_{-r}(\mathcal{N}_{j_0}))$.
If $a_j < \infty$ for some integer $j > j_0$, then
\begin{equation*}
x < x_{j-1} \cdot \sigma_{-r}(p_j^{a_j+1}) < x_{j-1} \cdot \frac1{1 - p_j^{-r}} .
\end{equation*}
Therefore, if $a_j < \infty$ for infinitely many integers $j > j_0$, then $x = \ell \in \Cl(\sigma_{-r}(\mathcal{N}_{j_0}))$.
Suppose instead that $a_j < \infty$ only for finitely many integers $j > j_0$.
Let $j_1$ be the least integer $\geq j_0$ such that $a_j = \infty$ for all integers $j > j_1$.
Thus we have
\begin{equation*}
x_j = x_{j_1} \cdot \prod_{k \, = \, j_1 + 1}^{j} \frac1{1 - p_k^{-r}} ,
\end{equation*}
for all integers $j > j_1$, so that
\begin{equation*}
\ell = x_{j_1} \cdot \prod_{k \,>\, j_1} \frac1{1 - p_k^{-r}} .
\end{equation*}
If $j_1 = j_0$, then 
\begin{equation*}
\prod_{k \,>\, j_0} \frac1{1 - p_k^{-r}} = \ell \leq x \leq \prod_{k \,>\, j_0} \frac1{1 - p_k^{-r}} ,
\end{equation*}
hence $x = \ell \in \Cl(\sigma_{-r}(\mathcal{N}_{j_0}))$.
If $j_1 > j_0$, then, by the minimality of $j_1$, we have $a_{j_1} < \infty$ so that
\begin{equation*}
x_{j_1} \cdot \prod_{k \,>\, j_1} \frac1{1 - p_k^{-r}} = \ell \leq x < x_{j_1-1} \cdot \sigma_{-r}(p_{j_1}^{a_{j_1}+1}) = x_{j_1} \cdot \frac{\sigma_{-r}(p_{j_1}^{a_{j_1}+1})}{\sigma_{-r}(p_{j_1}^{a_{j_1}})} \leq x_{j_1} \cdot (1 + p_{j_1}^{-r}) ,
\end{equation*}
but this is absurd since $\prod_{k > j_1} \frac1{1 - p_k^{-r}} \geq 1 + p_{j_1}^{-r}$.

So we have proved that $\left[1, \prod_{k > j_0} \frac1{1 - p_k^{-r}} \right] \subseteq \Cl(\sigma_{-r}(\mathcal{N}_{j_0}))$.
The other inclusion is immediate, since $\sigma_{-r}(\mathcal{N}_{j_0}) \subseteq \left[1, \prod_{k > j_0} \frac1{1 - p_k^{-r}} \right[$.
\end{proof}

Now we give a recursive formula to express $\Cl(\sigma_{-r}(\mathcal{N}_{j-1}))$, for a positive integer $j \leq j_0$, in terms of $\Cl(\sigma_{-r}(\mathcal{N}_j))$.

\begin{lem}\label{lem:rec}
For each positive integer $j \leq j_0$, we have
\begin{equation*}
\Cl(\sigma_{-r}(\mathcal{N}_{j-1})) = \bigcup_{a \,\in\, \mathbf{N} \,\cup\, \{\infty\}} \sigma_{-r}(p_j^a) \cdot \Cl(\sigma_{-r}(\mathcal{N}_j)) .
\end{equation*}
\end{lem}
\begin{proof}
Pick $x \in \Cl(\sigma_{-r}(\mathcal{N}_{j-1}))$, so that there exists a sequence $(n_k)_{k \geq 0}$ of elements of $\mathcal{N}_{j-1}$ such that $\sigma_{-r}(n_k) \to x$, as $k \to \infty$.
Let $a_k := \upsilon_{p_j}(n_k)$ and $m_k := n_k / p_j^{a_k}$, for each integer $k \geq 0$.
On the one hand, if $(a_k)_{k \geq 0}$ is bounded, then by passing to a subsequence of $(n_k)_{k \geq 0}$, we can assume that $a_k = a$ for all integers $k \geq 0$, for some $a \in \mathbf{N}$.
Therefore,
\begin{equation*}
x = \sigma_{-r}(p_j^a) \lim_{k \to \infty} \sigma_{-r}(m_k) \in \sigma_{-r}(p_j^a) \cdot \Cl(\sigma_{-r}(\mathcal{N}_j)) ,
\end{equation*}
since obviously $m_k \in \mathcal{N}_j$ for each integer $k \geq 0$.
On the other hand, if $(a_k)_{k \geq 0}$ is unbounded, then by passing to a subsequence of $(n_k)_{k \geq 0}$, we can assume that $a_k \to \infty$, as $k \to \infty$.
Therefore,
\begin{equation*}
x = \lim_{k \to \infty} \sigma_{-r}(p_j^{a_k}) \lim_{k \to \infty} \sigma_{-r}(m_k) \in \sigma_{-r}(p_j^\infty) \cdot \Cl(\sigma_{-r}(\mathcal{N}_j)) .
\end{equation*}
So we have proved that 
\begin{equation*}
\Cl(\sigma_{-r}(\mathcal{N}_{j-1})) \subseteq \bigcup_{a \,\in\, \mathbf{N} \,\cup\, \{\infty\}} \sigma_{-r}(p_j^a) \cdot \Cl(\sigma_{-r}(\mathcal{N}_j)) ,
\end{equation*}
the other inclusion is immediate.
\end{proof}

Finally, we need a technical lemma about certain unions of scaled intervals.

\begin{lem}\label{lem:techn}
Given $\beta > \alpha \geq 1$ and $j \in \mathbf{N}^+$, let $a_0$ be the least nonnegative integer such that
\begin{equation*}
\frac{\sigma_{-r}\big(p_j^{a_0+1}\big)}{\sigma_{-r}(p_j^{a_0})} \leq \frac{\beta}{\alpha} .
\end{equation*}
For each nonnegative integer $a < a_0$, put $J_a := \sigma_{-r}(p_j^a) \cdot [\alpha, \beta]$. 
Put also 
\begin{equation*}
J_{a_0} := [\sigma_{-r}(p_j^{a_0})\alpha, \sigma_{-r}(p_j^\infty)\beta] .
\end{equation*}
Then the intervals $J_0, \ldots, J_{a_0}$ are pairwise disjoint and it holds 
\begin{equation*}
\bigcup_{a \,\in\, \mathbf{N} \,\cup\, \{\infty\}} \sigma_{-r}(p_j^a) \cdot [\alpha, \beta] = J_0 \cup \cdots \cup J_{a_0} .
\end{equation*}
\end{lem}
\begin{proof}
The sequence $(\sigma_{-r}(p_j^{a+1}) / \sigma_{-r}(p_j^a))_{a \geq 0}$ is monotone nonincreasing and tends to $1$ as $a \to \infty$, while $\beta / \alpha > 1$.
Therefore, $a_0$ is well-defined and it holds $\sigma_{-r}(p_j^{a+1}) / \sigma_{-r}(p_j^a) \leq \beta / \alpha$, i.e., $\sigma_{-r}(p_j^{a+1}) \alpha \leq \sigma_{-r}(p_j^a) \beta$, for each integer $a \geq a_0$, which in turn implies
\begin{align*}
\bigcup_{\substack{a \,\in\, \mathbf{N} \,\cup\, \{\infty\} \\ a \,\geq\, a_0}} \sigma_{-r}(p_j^a) \cdot [\alpha, \beta] = [\sigma_{-r}(p_j^{a_0})\alpha, \sigma_{-r}(p_j^\infty)\beta] = J_{a_0} .
\end{align*}
Finally, by the minimality of $a_0$, we have $\sigma_{-r}(p_j^{a+1}) \alpha > \sigma_{-r}(p_j^a) \beta$ for each nonnegative integer $a < a_0$, hence $J_0, \ldots, J_{a_0}$ are pairwise disjoint.
The claim follows.
\end{proof}

\section{Proof of Theorem~\ref{thm:closure}}

We are now ready to prove Theorem~\ref{thm:closure}.
Actually, we shall prove the following slightly stronger lemma, from which Theorem~\ref{thm:closure} follows at once by choosing $j=0$.

\begin{lem}\label{lem:induction}
Fix any $r > 1$. 
Then for each nonnegative integer $j \leq j_0$, there exist: a positive integer $\ell_j$, pairwise disjoint closed intervals $I_{j,1}, \ldots, I_{j,\ell_j}$, and positive rationals $d_{j,1}, \ldots, d_{j,\ell_j}$, such that
\begin{equation}\label{equ:clj}
\Cl(\sigma_{-r}(\mathcal{N}_j)) = I_{j,1} \cup \ldots \cup I_{j,\ell_j} ,
\end{equation}
while
\begin{equation}\label{equ:dk}
d_{j,k} = \lim_{x \to \infty} \frac{\#\{n \in \mathcal{N}_j : n \leq x, \; \sigma_{-r}(n) \in I_{j,k}\}}{x} ,
\end{equation}
for $k=1,\ldots,\ell_j$.
\end{lem}
\begin{proof}
We proceed by (backward) induction on $j$.
For $j = j_0$, it is sufficient to take $\ell_{j_0} := 1$, 
\begin{equation*}
I_{j_0,1} := \left[1, \prod_{k \,>\, j_0} \frac1{1 - p_k^{-r}} \right] ,
\end{equation*}
and
\begin{equation*}
d_{j_0,1} := \lim_{x \to \infty} \frac{\#\{n \in \mathcal{N}_{j_0} : n \leq x\}}{x} = \prod_{k=1}^{j_0} \left(1 - \frac1{p_k}\right) ,
\end{equation*}
so that the claim follows by Lemma~\ref{lem:first_step}.

Now we assume to have proved the claim for $j \geq 1$, and we shall prove it for $j - 1$.
For $k=1,\ldots,\ell_j$, write $I_{j,k} = [\alpha_{j,k}, \beta_{j,k}]$ and let $a_{j,k}$ be the least nonnegative integer such that
\begin{equation*}
\frac{\sigma_{-r}\big(p_j^{a_{j,k}+1}\big)}{\sigma_{-r}\big(p_j^{a_{j,k}}\big)} \leq \frac{\beta_{j,k}}{\alpha_{j,k}}.
\end{equation*}
Then, put $J_{j,k,a} := \sigma_{-r}(p_j^a) \cdot I_{j,k}$ for all the nonnegative integers $a < a_{j,k}$ and also
\begin{equation*}
J_{j,k,a_{j,k}} := \left[\sigma_{-r}\big(p_j^{a_{j,k}}\big)\alpha_{j,k}, \sigma_{-r}(p_j^\infty)\beta_{j,k} \right] .
\end{equation*}
By Lemma~\ref{lem:techn}, it follows that $J_{j,k,0}, \ldots, J_{j,k,a_{j,k}}$ are pairwise disjoint and
\begin{equation}\label{equ:bigcupik}
\bigcup_{a \,\in\, \mathbf{N} \,\cup\, \{\infty\}} \sigma_{-r}(p_j^a) \cdot I_{j,k} = J_{j,k,0} \cup \cdots \cup J_{j,k,a_{j,k}} .
\end{equation}
At this point, from Lemma~\ref{lem:rec} and equations (\ref{equ:clj}) and (\ref{equ:bigcupik}), we obtain
\begin{align*}
\Cl(\sigma_{-r}(\mathcal{N}_{j-1})) &= \bigcup_{a \,\in\, \mathbf{N} \,\cup\, \{\infty\}} \sigma_{-r}(p_j^a) \cdot \Cl(\sigma_{-r}(\mathcal{N}_j)) = \bigcup_{a \,\in\, \mathbf{N} \,\cup\, \{\infty\}} \sigma_{-r}(p_j^a) \cdot \bigcup_{k=1}^{\ell_j} I_{j,k} \\
 &= \bigcup_{a \,\in\, \mathbf{N} \,\cup\, \{\infty\}} \bigcup_{k=1}^{\ell_j} \sigma_{-r}(p_j^a) \cdot I_{j,k} = \bigcup_{k=1}^{\ell_j} \bigcup_{a \,\in\, \mathbf{N} \,\cup\, \{\infty\}} \sigma_{-r}(p_j^a) \cdot I_{j,k} \\
 &= \bigcup_{k=1}^{\ell_j} (J_{j,k,0} \cup \cdots \cup J_{j,k,a_{j,k}}) = I_{j-1,1} \cup \cdots \cup I_{j-1,\ell_{j-1}},
\end{align*}
for some positive integer $\ell_{j-1}$ and some pairwise disjoint closed intervals \mbox{$I_{j-1,1}, \ldots, I_{j-1,\ell_{j-1}}$}.
In particular, note that for each $k \in \{1,\ldots,\ell_j\}$, each nonnegative integer $a \leq a_{j,k}$, and each $h \in \{1,\ldots,\ell_{j-1}\}$, it holds $J_{j,k,a} \subseteq I_{j-1,h}$, and in such a case we set $\delta_{j,k,a,h} := 1$; or it holds $J_{j,k,a} \cap I_{j-1,h} = \varnothing$, and in such other case we set $\delta_{j,k,a,h} := 0$.
Fix any $h \in \{1, \ldots, \ell_{j-1}\}$ and put for convenience
\begin{equation*}
\mathcal{S} := \{n \in \mathcal{N}_{j-1} : \sigma_{-r}(n) \in I_{j-1,h}\} ,
\end{equation*}
so that it remains only to prove that $\mathcal{S}$ has a positive rational asymptotic density $d_{j-1,k}$.
On~the one hand, any $n \in \mathcal{N}_{j-1}$ can be written in a unique way as $n = p_j^a m$, with $a \in \mathbf{N}$ and $m \in \mathcal{N}_j$.
On the other hand, by induction hypothesis $I_{j,1}, \ldots, I_{j,\ell_j}$ are pairwise disjoint and their union contains $\sigma_{-r}(\mathcal{N}_j)$.
Therefore, for any $x > 0$ we have
\begin{align*}
\#\{n \in \mathcal{S} : n \leq x\} &= \sum_{a=0}^\infty \#\!\left\{m \in \mathcal{N}_j : m \leq x / p_j^a, \; \sigma_{-r}(p_j^a) \sigma_{-r}(m) \in I_{j-1,h}\right\} \\
&= \sum_{k=1}^{\ell_j} \sum_{a=0}^\infty \#\!\left\{m \in \mathcal{N}_j : m \leq x / p_j^a, \; \sigma_{-r}(m) \in I_{j,k}, \;\sigma_{-r}(p_j^a) \sigma_{-r}(m) \in I_{j-1,h}\right\} .
\end{align*}
If $k \in \{1,\ldots,\ell_j\}$ and $m \in \mathcal{N}_j$ are such that $\sigma_{-r}(m) \in I_{j,k}$, then we have that $\sigma_{-r}(p_j^a) \sigma_{-r}(m) \in I_{j-1,h}$, for some $a \in \mathbf{N}$, if and only if $\delta_{j,k,\min\{a,a_{j,k}\},h} = 1$.
As a consequence,
\begin{align}\label{equ:long1}
\#\{n \in \mathcal{S} : n \leq x\} &= \sum_{k=1}^{\ell_j} \sum_{a=0}^\infty \delta_{j,k,\min\{a,a_{j,k}\},h} \#\!\left\{m \in \mathcal{N}_j : m \leq x / p_j^a, \; \sigma_{-r}(m) \in I_{j,k}\right\} \\
 &= \sum_{k=1}^{\ell_j} \left(\sum_{a=0}^{a_{j,k}-1} \delta_{j,k,a,h} \#\!\left\{m \in \mathcal{N}_j : m \leq x / p_j^a, \; \sigma_{-r}(m) \in I_{j,k}\right\} \right.  \nonumber \\
 &\phantom{====}+ \delta_{j,k,a_{j,k},h} \left.\sum_{a=a_{j,k}}^{\infty} \#\!\left\{m \in \mathcal{N}_j : m \leq x / p_j^a, \; \sigma_{-r}(m) \in I_{j,k}\right\}\right) \nonumber .
\end{align}
For any integer $A \geq \max\{a_{j,1}, \ldots, a_{j,\ell_j}\}$, since there are at most $x/p_j^{A+1}$ positive integers not exceeding $x$ and divisible by $p_j^{A+1}$, we can truncate (\ref{equ:long1}) getting
\begin{align}\label{equ:long2}
\# & \{n \in \mathcal{S} : n \leq x\} = \sum_{k=1}^{\ell_j} \left(\sum_{a=0}^{a_{j,k}-1} \delta_{j,k,a,h} \#\!\left\{m \in \mathcal{N}_j : m \leq x / p_j^a, \; \sigma_{-r}(m) \in I_{j,k}\right\} \right.\\
 &\phantom{====}+ \delta_{j, k, a_{j,k}, h} \left.\sum_{a=a_{j,k}}^A \#\!\left\{m \in \mathcal{N}_j : m \leq x / p_j^a, \; \sigma_{-r}(m) \in I_{j,k}\right\}\right) + O\!\left(\frac{x}{p_j^{A+1}}\right) \nonumber .
\end{align}
Thus, dividing (\ref{equ:long2}) by $x$, letting $x \to \infty$, and using (\ref{equ:dk}), we obtain
\begin{equation*}
\liminf_{x \to \infty }\frac{\#\{n \in \mathcal{S} : n \leq x\}}{x} = \sum_{k=1}^{\ell_j} d_{j,k} \left(\sum_{a=0}^{a_{j,k}-1} \frac{\delta_{j,k,a,h}}{p_j^a} + \delta_{j, k, a_{j,k}, h} \sum_{a=a_{j,k}}^A \frac1{p_j^a}\right) + O\!\left(\frac1{p_j^{A+1}}\right) ,
\end{equation*}
and similarly for the limit superior.
Therefore, as $A \to \infty$, it follows that
\begin{equation*}
\lim_{x \to \infty}\frac{\#\{n \in \mathcal{S} : n \leq x\}}{x} = \sum_{k=1}^{\ell_j} d_{j,k} \left(\sum_{a=0}^{a_{j,k}-1} \frac{\delta_{j,k,a,h}}{p_j^a} + \frac{\delta_{j, k, a_{j,k}, h}}{p_j^{a_{j,k}-1}(p_j-1)}\right) \in \mathbf{Q}^+ .
\end{equation*}
In conclusion, putting
\begin{equation}\label{equ:d_rec}
d_{j-1,h} := \sum_{k=1}^{\ell_j} d_{j,k} \left(\sum_{a=0}^{a_{j,k}-1} \frac{\delta_{j,k,a,h}}{p_j^a} + \frac{\delta_{j, k, a_{j,k}, h}}{p_j^{a_{j,k}-1}(p_j-1)}\right) ,
\end{equation}
for $h=1,\ldots,\ell_{j-1}$, we have proved the claim for $j-1$ and the proof is complete.
\end{proof}

\section{Algorithm to compute the closure of $\sigma_{-r}(\mathbf{N}^+)$}

Now we are ready to illustrate an algorithm to exactly compute the closure of $\sigma_{-r}(\mathbf{N}^+)$, given any $r > 1$.
Well, the words ``algorithm'' and ``exactly'' need some explanation.
We assume our algorithm runs on a machine always able to decide inequalities involving finite products of the quantities $\sigma_{-r}(p_j^a)$ and $\prod_{k \,>\, j} \frac1{1-p_k^{-r}}$, with $j \in \mathbf{N}^+$ and $a \in \mathbf{N} \cup \{\infty\}$.
For the most $r$, probably almost all $r$ with respect to Lebesgue measure, inequalities of that kind can be decided by numerical computation, using (\ref{equ:crucial_prod}).
However, in general those inequalities are probably undecidable (how to exclude that the equality holds?), although this poses no practical problems.
At the end of its execution, the algorithm returns the extreme points of the intervals $I_1, \ldots, I_\ell$ in the statement of Theorem~\ref{thm:closure}, as finite products of the quantities above, and also $d_1, \ldots, d_\ell$.

The idea behind the algorithm is just to make constructive the proof of Lemma~\ref{lem:induction}, we give only the essential details.

\newpage
\begin{algorithmic}[1]
\Require{$r > 1$.}
\Ensure{
A positive integer $\ell_0$, pairwise finite closed intervals $I_{0,1}, \ldots, I_{0,\ell_0}$, and positive rational numbers $d_{0,1}, \ldots, d_{0,\ell_0}$, such that
\begin{equation*}
\Cl(\sigma_{-r}(\mathbf{N}^+)) = I_{0,1} \cup \cdots \cup I_{0,\ell_0} ,
\end{equation*}
while
\begin{equation*}
d_{0,k} = \lim_{x \to \infty} \frac{\#\{n \in \mathbf{N}^+ : n \leq x, \; \sigma_{-r}(n) \in I_{0,k}\}}{x} ,
\end{equation*}
for $k=1,\ldots,\ell_0$.
}
\setstretch{1.3}
\State Compute $j_0$ as explained just after the proof of Lemma~\ref{lem:prod_bound}.
\State $\ell_{j_0} \leftarrow 1$
\State $I_{j_0,1} \leftarrow \left[1, \prod_{k > j_0} \frac1{1 - p_k^{-r}}\right]$
\State $d_{j_0,1} \leftarrow \prod_{k=1}^{j_0}\left(1-\frac1{p_k}\right)$
\For{$j = j_0, j_0 - 1, \ldots, 1$}
 \For{$k = 1,\ldots,\ell_j$}
 \State $a_{j,k} \leftarrow \min\!\big\{a \in \mathbf{N} : \sigma_{-r}(p_j^{a+1}) / \sigma_{-r}(p_j^a) \leq \sup(I_{j,k}) / \inf(I_{j,k})\big\}$
  \For{$a = 0,\ldots,a_{j,k}-1$}
   \State $J_{j,k,a} \leftarrow \sigma_{-r}(p_j^a) \cdot I_{j,k}$
  \EndFor
  \State $J_{j,k, a_{j,k}} \leftarrow [\sigma_{-r}(p_j^a)\inf(I_{j,k}), \sigma_{-r}(p_j^\infty)\sup(I_{j,k})]$
 \EndFor
 \State Compute pairwise disjoint closed intervals $I_{j-1,1}, \ldots, I_{j-1,\ell_{j-1}}$ such that
 \begin{equation*}
 \bigcup_{k=1}^{\ell_j} (J_{j, k, 0} \cup \cdots \cup J_{j, k, a_{j,k}}) = I_{j-1,1} \cup \cdots \cup I_{j-1,\ell_{j-1}} .
 \end{equation*}
 \For{$k = 1, \ldots, \ell_j$, $a = 0, \ldots, a_{j,k}$, $h=1,\ldots,\ell_{j-1}$}
  \If{$J_{j, k, a} \subseteq I_{j-1,h}$}
   \State $\delta_{j, k, a, h} \leftarrow 1$
  \Else
   \State $\delta_{j, k, a, h} \leftarrow 0$
  \EndIf
 \EndFor
 \For{$h = 1,\ldots,\ell_{j-1}$}
  \State $d_{j-1,h} \leftarrow \sum_{k=1}^{\ell_j} d_{j,k} \left(\sum_{a=0}^{a_{j,k}-1} \frac{\delta_{j,k,a,h}}{p_j^a} + \frac{\delta_{j, k, a_{j,k}, h}}{p_j^{a_{j,k}-1}(p_j-1)}\right)$.
 \EndFor
\EndFor
\end{algorithmic}
\vskip 1em

The author implemented this algorithm in the Python programming language, with floating point arithmetic.
Running the implementation on a personal computer, he computed $\Cl(\sigma_{-r}(\mathbf{N}^+))$ from $r = 1.5$ to $3.5$ with a step of $0.01$, leading to the plot of Figure~\ref{fig:plot}.

\begin{figure}[h]
\begin{center}
\includegraphics[scale=4.0]{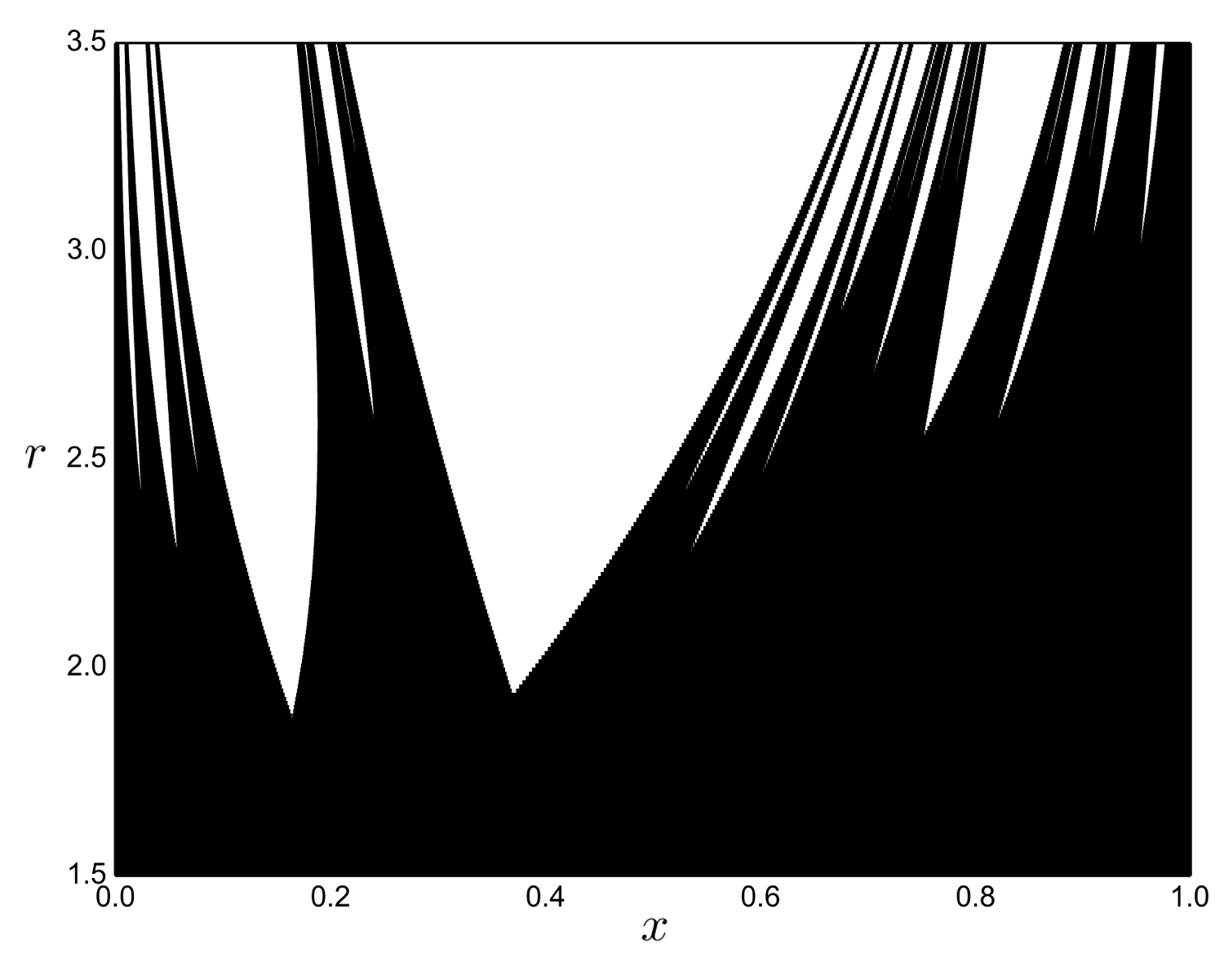}
\caption{
A point $(x, r)$ of the plot is black if and only if $(\zeta(r) - 1)x + 1 \in \Cl(\sigma_{-r}(\mathbf{N}^+))$.
}
\label{fig:plot}
\end{center}
\end{figure}

\section{Closure of $\sigma_{-2}(\mathbf{N}^+)$}
In this section, we compute the closure of $\sigma_{-2}(\mathbf{N}^+)$.
We shall use the same notation of the proof of Lemma~\ref{lem:induction}.
First, in light of Lemma~\ref{lem:next_prime_bound} a quick computation shows that $p_{j+1} \leq 2^{1/2} p_j$ for all positive integers $j > 4$.
Moreover,
\begin{equation*}
\prod_{k \, > \, 1} \frac1{1 - p_k^{-2}} < 1 + p_1^{-2}, \;\; \prod_{k \, > \, 2} \frac1{1 - p_k^{-2}} < 1 + p_2^{-2},\; \text{ and } \prod_{k \, > \, 3} \frac1{1 - p_k^{-2}} > 1 + p_3^{-2} , 
\end{equation*}
hence $j_0 = 2$.
Thus, $\ell_2 = 1$, 
\begin{equation*}
I_{2,1} = \left[1, \prod_{k \,>\, 2} \frac1{1 - p_k^{-2}} \right] = \left[1, \left(1-2^{-2}\right)\cdot \left(1-3^{-2}\right)\cdot \zeta(2) \right] = \left[1, \frac{\pi^2}{9} \right] ,
\end{equation*}
and
\begin{equation*}
d_{2,1} = \prod_{k=1}^{2} \left(1 - \frac1{p_k}\right) = \left(1 - \frac1{2}\right)\cdot \left(1 - \frac1{3}\right) = \frac1{3} .
\end{equation*}
Now $\sigma_{-2}(3^1) / \sigma_{-2}(3^0) > \pi^2/9$ and $\sigma_{-2}(3^2) / \sigma_{-2}(3^1) < \pi^2/9$, hence $a_{2,1} = 1$, so that
\begin{equation*}
J_{2,1,0} = \sigma_{-2}(3^0) \cdot I_{2,1} = I_{2,1} = \left[1, \frac{\pi^2}{9}\right],
\end{equation*}
and
\begin{equation*}
J_{2,1,1} = \left[\sigma_{-2}(3^1), \sigma_{-2}(3^\infty)\cdot \frac{\pi^2}{9}\right] = \left[\frac{10}{9}, \frac{\pi^2}{8}\right] .
\end{equation*}
Therefore, $\ell_1 = 2$ and we can take $I_{1,1} = J_{2,1,0}$ and $I_{1,2} = J_{2,1,1}$, so that by (\ref{equ:d_rec}) we have
\begin{equation*}
d_{1,1} = d_{2,1} \cdot \frac1{p_2^0} = \frac1{3} , \text{ and } d_{1,2} = d_{2,1} \cdot \frac1{p_2^0 (p_2-1)} = \frac1{6} .
\end{equation*}
Now $\sigma_{-2}(2^1) / \sigma_{-2}(2^0) > \pi^2/9$ and $\sigma_{-2}(2^2) / \sigma_{-2}(2^1) < \pi^2/9$, hence $a_{1,1} = 1$, so that
\begin{equation*}
J_{1,1,0} = \sigma_{-2}(2^0) \cdot I_{1,1} = I_{1,1} = \left[1, \frac{\pi^2}{9}\right] ,
\end{equation*}
and
\begin{equation*}
J_{1,1,1} = \left[\sigma_{-2}(2^1), \sigma_{-2}(2^\infty)\cdot \frac{\pi^2}{9}\right] = \left[\frac{5}{4}, \frac{4\pi^2}{27}\right] .
\end{equation*}
Also, $\sigma_{-2}(2^1) / \sigma_{-2}(2^0) > \tfrac{\pi^2}{8} / \tfrac{10}{9}$ and $\sigma_{-2}(2^2) / \sigma_{-2}(2^1) < \tfrac{\pi^2}{8} / \tfrac{10}{9}$, hence $a_{1,2} = 1$, so that
\begin{equation*}
J_{1,2,0} = \sigma_{-2}(2^0) \cdot I_{1,2} = I_{1,2} = \left[\frac{10}{9}, \frac{\pi^2}{8}\right] ,
\end{equation*}
and
\begin{equation*}
J_{1,2,1} = \left[\sigma_{-2}(2^1) \cdot \frac{10}{9}, \sigma_{-2}(2^\infty)\cdot \frac{\pi^2}{8}\right] = \left[\frac{25}{18}, \frac{\pi^2}{6}\right] .
\end{equation*}
The union of $J_{1,1,0}$, $J_{1,1,1}$, $J_{1,2,0}$, and $J_{1,2,1}$ has $\ell_0 = 3$ connected components, namely
\begin{equation*}
I_{0,1} = J_{1,1,0} = \left[1, \frac{\pi^2}{9}\right], \;\; I_{0,2} = J_{1,2,0} = \left[\frac{10}{9}, \frac{\pi^2}{8}\right] , 
\end{equation*}
and
\begin{equation*}
I_{0,3} = J_{1,1,1} \cup J_{1,2,1} = \left[\frac{5}{4}, \frac{\pi^2}{6}\right] .
\end{equation*}
Lastly, by (\ref{equ:d_rec}) it follows
\begin{equation*}
d_{0,1} = d_{1,1} \cdot \frac1{p_1^0} = \frac1{3}, \;\; d_{0,2} = d_{1,2} \cdot \frac1{p_1^0} = \frac1{6} ,
\end{equation*}
and
\begin{equation*}
d_{0,3} = d_{1,1} \cdot \frac1{p_1-1} + d_{1,2}\cdot \frac1{p_1-1} = \frac1{2} .
\end{equation*}
The claim about $\Cl(\sigma_{-2}(\mathbf{N}^+))$ given in the introduction is proved.

\section{Concluding remarks}

It seems likely that the proof of Theorem~\ref{thm:closure} could be adapted to multiplicative arithmetic functions $f : \mathbf{N}^+ \to \mathbf{R}^+$ other than $\sigma_{-r}$, satisfying at least the following hypothesis:
\begin{enumerate}[(i).]
\item For each prime number $p$, the sequence $(f(p^a))_{a \geq 0}$ is monotone nondecreasing and 
\begin{equation*}
f(p^\infty) := \lim_{a \to \infty} f(p^a)
\end{equation*}
exists finite.
\item For each prime number $p$, the sequence $(f(p^{a+1})/f(p^a))_{a \geq 0}$ is monotone nonincreasing.
\item It holds 
\begin{equation*}
\prod_{k \, > \, j} f(p_k^\infty) < f(p_j),
\end{equation*}
only for finitely many positive integers $j$.
\end{enumerate}

\subsection*{Acknowledgments}

The author thanks Colin~Defant (University of Florida) for his careful reading of the paper and useful comments.

\bibliographystyle{amsalpha}
\bibliography{closure}

\end{document}